\documentclass[a4paper,10pt]{amsart}

\numberwithin{equation}{section}
\usepackage[top=3.5cm,bottom=3.5cm,left=3.6cm,right=3.6cm]{geometry}
\usepackage{amsfonts}
\usepackage{amsmath}
\usepackage{amssymb}
\usepackage{amsthm}
\usepackage[utf8x]{inputenc}
\usepackage[T1]{fontenc}
\usepackage{graphicx}
\usepackage{hyperref}
\usepackage{enumerate}
\usepackage{mathtools}
\usepackage{enumitem}

\newtheorem{dfn}{Definition}[section]
\newtheorem{propo}[dfn]{Proposition}

\newtheorem{theo}[dfn]{Theorem}

\newtheorem{cor}[dfn]{Corollary}
\newtheorem{lem}[dfn]{Lemma}

\theoremstyle{definition}

\theoremstyle{remark}

\newtheorem{rem}[dfn]{Remark}

\newcommand{\norm}[1]{\left\lVert #1 \right\rVert}
\newcommand{\ens}[1]{\left\{ #1\right\}}
\newcommand{\R}{\mathbb{R}}
\newcommand{\Z}{\mathbb{Z}}

\newcommand{\abs}[1]{\left|#1\right|}

\newcommand{\itlog}[1]{#1\log\log #1}  
\providecommand{\newoperator}[3]{%
  \newcommand*{#1}{\mathop{#2}#3}}
\newoperator{\re}{\mathrm{Re}}{\,}
\newoperator{\im}{\mathrm{Im}}{\,}

\begin{document}

\title{Integrability conditions on coboundary and transfer function for limit theorems}

\author{Davide Giraudo }

\address{Normandie Universit\'e, Universit\'e de Rouen, 
Laboratoire de Math\'ematiques Rapha\"el Salem,
CNRS, UMR 6085, Avenue de l'universit\'e, BP 12, 
76801 Saint-Etienne du Rouvray Cedex, 
France.}

\email{davide.giraudo1@univ-rouen.fr}

\date{\today}

\keywords{Invariance principle, strictly stationary process}

\subjclass[2010]{60F17; 60G10}

 \begin{abstract}
  For a measure preserving automorphism $T$ of a probability space, we provide 
  conditions on the tail function of $g\colon\Omega\to\R$ 
  and $g-g\circ T$ which guarantee limit theorems among the weak invariance principle, 
  Marcinkievicz-Zygmund strong law of large numbers and the law of the iterated logarithm 
  to hold for $f:=m+g-g\circ T$, where $(m\circ T^i)_{i\geqslant 0}$ 
  is a martingale differences sequence.
 \end{abstract}

\maketitle 

\section{Introduction and notations}
 
Let $(\Omega,\mathcal F,\mu)$ be a probability space and $T\colon \Omega\to \Omega$ 
be a bijective bi-measurable and measure preserving map. We assume that the dynamical 
system is ergodic (that is, if $T^{-1}A=A$ for some $A\in\mathcal F$, then 
$\mu(A)\in \ens{0,1}$).
If $n\geqslant 1$ is an integer and $f\colon\Omega\to\R$, we denote $S_n(f):=
\sum_{j=0}^{n-1}f\circ T^j$ and for a fixed $t$, define 
\begin{equation}
 S_n^{\mathrm{pl} }(f,t):=S_{\left[nt\right]}(f)+(nt-\left[nt\right])  
 f\circ T^{\left[nt\right]}, t\in [0,1], 
\end{equation}
 where $\left[x\right]$ denote the integer part of the real number $x$.
Then for each $\omega\in \Omega$ and each integer $n\geqslant 1$, 
the map $t\mapsto S_n^{\mathrm{pl} }(f,t)$ is 
an element of the space of continuous functions on $[0,1]$, denoted 
by $C[0,1]$.  
 
Let us state the limit theorems we are interested in. 

\begin{dfn}Let $f\colon\Omega\to\R$ be a measurable function.
\begin{itemize}
 \item We say that the function $f$ satisfies the invariance principle if the 
 sequence $(n^{-1/2}S_n^{\mathrm{pl} }(f,\cdot) )_{n\geqslant 1}$ weakly 
 converges in the space $C[0,1]$ endowed with the topology of uniform 
 convergence to a scalar multiple of a standard Brownian motion.
 \item We say that the function $f$ satisfies the law of the iterated logarithm 
 if there exists a constant $C(f)$ such that for almost every $\omega\in\Omega$, 
 \begin{equation}\label{eq:def_LIL}
  \limsup_{n\to +\infty }\frac{S_n(f)(\omega)}{\sqrt{\itlog n}}=C(f)\mbox{ and }  
  \liminf_{n\to +\infty }\frac{S_n(f)(\omega)}{\sqrt{\itlog n}}=-C(f).  
 \end{equation}
 \item We say that the function $f$ satisfies the functional law of the
 iterated logarithm if the sequence $\left( (\sqrt{\itlog{n}})^{-1}
 S_n^{\mathrm{pl} }(f,\cdot)\right)_{n\geqslant 1} $ is relatively 
 compact and the set of its limit points coincides with the set of all 
 absolutely continuous functions $x\in C[0,1]$ such that $x(0)=0$ and 
 $\int_0^1 (x'(t))^2\mathrm dt\leqslant 1$, where $x'$ denotes the 
 derivative of $x$ defined almost everywhere with respect to the Lebesgue measure.
 \item Let $1<p<2$. We say that the function $f$ satisfies the 
 $p$-strong law of large numbers if for any $\alpha\in[1/p,1]$, 
 the following holds:
 \begin{equation}\label{Baum-Katz}
 \forall\varepsilon>0,\quad   \sum_{n=1}^{+\infty}n^{\alpha p-2}
 \mu\ens{\max_{1\leqslant k\leqslant n} 
 \abs{S_k(f)}\geqslant\varepsilon n^\alpha }< +\infty. 
 \end{equation}
\end{itemize}
\end{dfn}

If it is possible to find a decomposition of the function $f$ 
\begin{equation}\label{decomposition_f}
 f=m+g- g\circ T,
\end{equation}
where $g\colon\Omega\to\R$ is a measurable function and $m$ 
satisfies one of the previous limit theorems, then one can wonder 
if we can deduce the result for $f$. 

A known situation is when the sequence $(m\circ T^i)_{i\geqslant 0}$ is 
a square-integrable martingale differences sequence. A necessary and 
sufficient condition to have \eqref{decomposition_f} with such an $m$ and 
a square integrable $g$ is known (see Theorem~2 in \cite{MR1198662}). 
If $(m\circ T^i)_{i\geqslant 0} $ is a square-integrable martingale 
difference sequence, then the functional law of the iterated logarithm and 
the invariance principle take place. 

If $1<p<2$ and $m\in\mathbb L^p$, then Theorem~5 by Dedecker and 
Merlev\`ede \cite{MR2743029} 
implies that $m$ satisfies the $p$-strong law of large numbers.
Actually, their results holds in a more general setting than 
strictly stationary sequences, as they only require a stochastic 
domination on the considered martingale differences sequence. 
A similar result as \eqref{Baum-Katz} takes place for $\alpha=1$
if we require a conditional stochastic domination (see \cite{benoist_quint}, 
Theorem~2.2). A necessary and sufficient condition 
for \eqref{decomposition_f} to hold with $m,g\in\mathbb L^p$, $1<p<2$,  is given by Voln\'y 
in Theorem~1 of \cite{MR2239087}, and in this case, \eqref{Baum-Katz} is 
satisfied (see Theorem~6 of \cite{MR2743029}).

We call a \textit{coboundary} a function of the form $g-g\circ T$, where 
$g\colon\Omega\to \R$ is a measurable function. The function $g$ is called 
a \textit{transfer function}. The following result is Theorem~1 
of \cite{MR1893125}. It gives a necessary and sufficient condition 
on the transfer function to preserve the 
limit theorems mentioned in the previous definition. Sufficiency for the 
invariance principle and the law of the iterated logarithm 
was established in \cite{MR624435}, pages 140-142).

 \begin{theo}[The equivalence theorem,\cite{MR1893125}]\label{equivalence_theorem}
 Let us suppose that for the process $(m\circ T^i)_{i \in \Z}$ the invariance 
 principle, the law of the iterated logarithm (functional law of the iterated 
 logarithm) respectively, holds true. Let $g$ be a measurable function and 
 \begin{equation}
  f=m+g - g\circ T.
 \end{equation}
 Then for the process $(f\circ T^i)_{i \in \Z}$ 
 \begin{itemize} 
  \item the invariance principle holds if and only if 
  \begin{equation}\label{equivalence_WIP}
   \frac 1{\sqrt n}\max_{1\leqslant k\leqslant n}\abs{g\circ T^k}\underset{n\to \infty}{\to}  0
   \mbox{ in probability;} 
  \end{equation}
  \item the law of the iterated logarithm as well as the functional law of the iterated logarithm 
  holds if and only if 
   \begin{equation}\label{equivalence_LIL}
    \frac 1{\sqrt{n\log\log n}}g\circ T^n \underset{n\to \infty}{\to}  0\mbox{ a.s.}  
   \end{equation}
 \end{itemize}
 \end{theo}

Both conditions \eqref{equivalence_WIP} and \eqref{equivalence_LIL} take place 
when the function $g$ is square-integrable. If $1<p<2$, Theorem~6 
in \cite{MR2743029} shows that \eqref{Baum-Katz} holds 
if $g$ belongs to $\mathbb L^p$.

However, it may happen that 
we obtain a decomposition \eqref{decomposition_f} where $m\in\mathbb L^2$ 
but the function $g$ is only integrable (see \cite{MR1893125} for 
explicit counter-examples, 
and \cite{MR780729} for a condition which guarantees the 
square integrability of $m$) and in this case, 
the weak invariance principle does not need to hold. We investigate
conditions on the functions $t\mapsto \mu\ens{\abs g >t}$ and   
$t\mapsto \mu\ens{\abs{g-g\circ T}  >t}$ which guarantee \eqref{equivalence_WIP}
or \eqref{equivalence_LIL}. In order to state these conditions in a more 
concise way, we introduce the so-called weak $\mathbb L^q$-spaces.
 
\begin{dfn}
 Let $q$ be a real number strictly greater than $1$. We denote by 
 $\mathbb L^{q,\infty}$ the space of functions $h\colon\Omega\to R$ 
 such that 
 \begin{equation}
   \norm h_{q,\infty}^q :=\sup_{t>0}t^q\mu\ens{\abs h>t}\mbox{ is finite}.   
 \end{equation}
 The subspace of $\mathbb L^{q,\infty}$ which consists of functions 
 $h$ such that $\displaystyle\lim_{t\to +\infty} t^q\mu\ens{\abs h>t}=0$ is 
 denoted by $\mathbb L^{q,\infty}_0$.  
\end{dfn}

\section{Main results}

 In this section, we state the main results of this note. In the first 
 subsection, we give a sufficient condition on the 
 functions $t\mapsto \mu\ens{\abs g >t}$ and   
$t\mapsto \mu\ens{\abs{g-g\circ T}  >t}$
in order to preserve the weak invariance 
principle, the law of the iterated 
logarithm and the $p$-strong law of large numbers respectively.
We provide projective conditions which guarantee these sufficient 
conditions an a martingale coboundary decomposition. 

In the second subsection, we construct counter-examples which 
show that the found conditions are sharp when the considered 
dynamical system is aperiodic.

Finally, in the third subsection, we provide applications 
of the results of Subsection~\ref{sec:sufficient_conditions} 
to Bernoulli shifts.

 Voln\'y and Samek showed in \cite{MR1893125} that the conclusion of
 Theorems~\ref{cond_suff_WIP} and \ref{cond_suff_LIL} (see the 
 next subsections) holds when
 $p\geqslant (r+2)/r$ and that of Theorem~\ref{counter_example} 
 when $p<(r-1)/(r-3/2)$. In the case  $r>2$, we cannot 
 conclude from their results if $(r-1)/(r-3/2)\leqslant  p<(r+2)/r$, 
 while it is the case with our conditions.
 
\subsection{Sufficient conditions}\label{sec:sufficient_conditions}

\begin{theo}\label{cond_suff_WIP} 
 Let $1< p<2$ and 
 let $g\colon\Omega\to\R$ be a function such that $g \in\mathbb L^{p,\infty}_0 $ and 
 $g-g\circ T \in\mathbb L^{p/(p-1),\infty}_0 $. Then for any square integrable 
 martingale differences sequence 
 $(m\circ T^i)_{i\geqslant 0}$, the function $f:= m+g-g\circ T$ 
 satisfies the weak invariance principle in $C[0,1]$.
\end{theo}
 
 A similar result has been obtained for the quenched functional central limit theorem 
 (see \cite{1503.05532}, Corollary~7).

\begin{cor}\label{cor:WIP}
 Let $1<p<2$ and let $\mathcal M$ be a sub-$\sigma$-algebra of $\mathcal F$ such that 
 $T\mathcal M\subset\mathcal M$. Assume that $f$ is an $\mathcal M$-measurable element 
 of $\mathbb L^{p/(p-1),\infty}_0$ such that the following two conditions hold:
 \begin{equation}\label{eq:convergence_E0(Sn)}
  \mbox{ the sequence }\left(\mathbb E\left[S_n(f)\mid\mathcal M\right]\right)_{n\geqslant 1}
  \mbox{ converges in }\mathbb L^{p,\infty};
 \end{equation}
 \begin{equation}\label{eq:convergence_(I-U)E0(Sn)}
  \mbox{ the sequence }\left((I-U)\mathbb E\left[S_n(f)\mid\mathcal M\right]\right)_{n\geqslant 1}
  \mbox{ converges in }\mathbb L^{p/(p-1),\infty}.
 \end{equation}
 Then the function $f$ satisfies the invariance principle.
 
\end{cor}

 The conditions of Corollary~\ref{cor:WIP} imply that $f$ admit the 
 martingale-coboundary representation \eqref{decomposition_f} with $m$ and 
 $g$ integrable and $g$ satisfies \eqref{equivalence_WIP}. In \cite{MR2446326,MR2475604},
 the later condition was compared with Dedecker and Rio 
 projective criterion (cf. \cite{MR1743095}):
 \begin{equation}\label{eq:Dedecker_and_Rio_condition}
  \mbox{the sequence }f\mathbb E[S_n(f)\mid\mathcal M]\mbox{ converges in }
  \mathbb L^1,
 \end{equation}
 which also implies that $f$ satisfies the invariance principle.
 It was shown that there is an example of function $f$ which satisfies the 
 martingale-coboundary decomposition in $\mathbb L^1$ and the invariance principle 
 but not  \eqref{eq:Dedecker_and_Rio_condition}.
 
 Here do not assume that $f$ belongs to $\mathbb L^{p/(p-1)}$ and that the convergence in 
 \eqref{eq:convergence_(I-U)E0(Sn)} holds in $\mathbb L^{p}$, otherwise, 
 the function $f$ would satisfy Dedecker-Rio projective criterion. 
 In this way, our condition is independent of \eqref{eq:Dedecker_and_Rio_condition}.

 \begin{propo}\label{propo:comparison}
  Let $(\Omega,\mathcal F,\mu)$ be a dynamical system of positive entropy. 
  Let $1<p<2$. There exists a sub-$\sigma$-algebra $\mathcal M$ such that 
 $T\mathcal M\subset\mathcal M$ and a function $f\in\mathbb L^{p/(p-1),\infty}_0$ satisfying
 \eqref{eq:convergence_E0(Sn)} and \eqref{eq:convergence_(I-U)E0(Sn)} but not 
 \eqref{eq:Dedecker_and_Rio_condition}.
 \end{propo}

\begin{theo}\label{cond_suff_LIL}
Let $1<p<2< r$ and let $g\colon\Omega\to\R$ be a function. 
\begin{enumerate}[label=(\roman*)]
 \item\label{strict_inequality}  If $p>r/(r-1)$, $g \in\mathbb L^{p,\infty}$ and 
 $g-g\circ T \in\mathbb L^{r,\infty}$, then for any martingale differences sequence 
 $(m\circ T^i)_{i\geqslant 0}$, the function 
 $f:= m+g-g\circ T$ satisfies the law of the iterated logarithm;
 \item\label{equality_case}  if $p=r/(r-1)$, $g \in\mathbb L^p $ and 
 $g-g\circ T \in\mathbb L^r$, then for any square integrable martingale differences sequence 
 $(m\circ T^i)_{i\geqslant 0}$, the function 
 $f:= m+g-g\circ T$ satisfies the law of the iterated logarithm.
\end{enumerate}
 
\end{theo} 
  
\begin{cor}\label{cor:cond_suff_LIL}
 Let $1<p<2$ and let $\mathcal M$ be a sub-$\sigma$-algebra of $\mathcal F$ such that 
 $T\mathcal M\subset\mathcal M$. Assume that $f$ is an $\mathcal M$-measurable element 
 of $\mathbb L^{p/(p-1)}$ such that the following two conditions hold:
 \begin{equation}\label{eq:convergence_E0(Sn)_LIL}
  \mbox{ the sequence }\left(\mathbb E\left[S_n(f)\mid\mathcal M\right]\right)_{n\geqslant 1}
  \mbox{ converges in }\mathbb L^{p};
 \end{equation}
 \begin{equation}\label{eq:convergence_(I-U)E0(Sn)_LIL}
  \mbox{ the sequence }\left((I-U)\mathbb E\left[S_n(f)\mid\mathcal M\right]\right)_{n\geqslant 1}
  \mbox{ converges in }\mathbb L^{p/(p-1)}.
 \end{equation}
 Then the function $f$ satisfies the functional law of the iterated logarithms.
\end{cor}

 \begin{rem}  \label{rmq:weak_Lp}
  In \cite{MR1893125}, $\mathbb L^q$ spaces are involved. It turns out 
  that in the setting of Theorems~\ref{cond_suff_WIP} and \ref{cond_suff_LIL}, 
  \ref{strict_inequality}, we may work with weak $\mathbb L^q$-spaces. 
  For the case \ref{equality_case} in Theorem~\ref{cond_suff_LIL}, 
  it is an open question to determine whether strong moments 
  are actually needed.
 \end{rem}
 \begin{theo}\label{cond_suff_LLN}
 Let $1\leqslant q<p<r<2$ be real numbers and let $g\colon\Omega\to \R$ 
 be a function such that $g \in\mathbb L^q$ and $g-g\circ T\in\mathbb L^r$. 
 If $q\geqslant (p-1)r /(r-1)$, then for any martingale differences sequence 
 $(m\circ T^i)_{i\geqslant 0}$ with $m\in \mathbb L^p$, the function $f:= m+g-g\circ T$ 
 satisfies \eqref{Baum-Katz}. 
\end{theo}

\begin{cor}\label{cor:BK}
 Let $1<p<r<2$ and let $\mathcal M$ be a sub-$\sigma$-algebra of $\mathcal F$ such that 
 $T\mathcal M\subset\mathcal M$. Assume that $f$ is an $\mathcal M$-measurable element 
 of $\mathbb L^{r}$ such that the following two conditions hold:
 \begin{equation}\label{eq:convergence_E0(Sn)_BK}
  \mbox{ the sequence }\left(\mathbb E\left[S_n(f)\mid\mathcal M\right]\right)_{n\geqslant 1}
  \mbox{ converges in }\mathbb L^{\max\ens{1,(p-1)r /(r-1)}};
 \end{equation}
 \begin{equation}\label{eq:convergence_(I-U)E0(Sn)_BK}
  \mbox{ the sequence }\left((I-U)\mathbb E\left[S_n(f)\mid\mathcal M\right]\right)_{n\geqslant 1}
  \mbox{ converges in }\mathbb L^{r}.
 \end{equation}
 Then for each positive $\varepsilon$ and any $\alpha\in[1/p,1]$, we have 
  \begin{equation}\label{eq:BK_rewritten}
 \sum_{n=1}^{+\infty}n^{\alpha p-2}
 \mu\ens{\max_{1\leqslant k\leqslant n} 
 \abs{S_k(f)}\geqslant\varepsilon n^\alpha }< +\infty. 
 \end{equation}
\end{cor}

\begin{rem}
 Ergodicity of the dynamical system is required for the "only if" direction 
 in the equivalence involving the law of the iterated logarithm of
 Theorem~\ref{equivalence_theorem}. Therefore, the results of this subsection 
 remain valid in the non-ergodic setting.
\end{rem}

\subsection{Counter-examples} 
 
 In Theorems~\ref{cond_suff_WIP} and \ref{cond_suff_LIL}, we gave a sufficient 
 condition on $(p,r)$ for which $g\in \mathbb L^p$ and $g-g\circ T^r\in\mathbb L^r$ 
 guarantees the invariance principle, and the law of iterated logarithms for 
 $g-g\circ T$. 
 The next results show that when this condition is not satisfied, 
 these limit theorems may fail, which shows its sharpness.
 
\begin{theo}\label{counter_example} 
Assume that the dynamical system $(\Omega,\mathcal F,\mu,T)$ is 
aperiodic. Let $1\leqslant p<2\leqslant r$ be real numbers such that $p<r/(r-1)$. 
Then there exists a function $g\in \mathbb L^p$ such that $g-g\circ T \in\mathbb L^r$ 
 and the function $g-g\circ T$ satisfies neither the invariance 
 principle nor the law of the iterated logarithm.
\end{theo}
 
We also have a similar counter-example for the $p$-strong law of large numbers.
 
\begin{theo}\label{counter_example_LLN}
 Assume that the dynamical system $(\Omega,\mathcal F,\mu,T)$ is 
aperiodic. Let $1<p<2$ and let $1\leqslant q<p<r$ be real numbers such that 
 $q<(p-1)r/(r-1)$. Then there exists a function $g\in\mathbb L^q$ such that 
 $g-g\circ T\in\mathbb L^r$ but the sequence 
 $(n^{-1/p} S_n(g-g\circ T))_{n\geqslant 1}$ does not converge almost surely to 
 $0$. In particular, $g-g\circ T$ does not satisfy \eqref{Baum-Katz}  
\end{theo}
 
 \subsection{Applications}\label{sec:applications}
 
In Subsection~\ref{sec:sufficient_conditions}, we provided sufficient 
conditions for the functional central limit theorem, the law of the iterated 
logarithm and the $p$-law of large numbers. It is natural to try to apply these 
conditions to some strictly stationary sequences. 

In the case 
of a linear process $f\circ T^k=\sum_{i\geqslant 0}a_i\varepsilon_{k-i}$, where 
$(\varepsilon_i)_{i\in \Z}$ is an i.i.d. centered sequence of random variables in 
$\mathbb L^{p/(p-1),\infty}_0$, $1<p<2$, condition 
\eqref{eq:convergence_E0(Sn)} implies that $f$ has a martingale coboundary decomposition 
in $\mathbb L^q$, $1<q<p$. Since $q/2<1$, the Marcinkievicz-Zygmund and Jensen's inequalities give
the convergence of the series $\sum_{i\geqslant 0}\left(\sum_{k\geqslant i}a_k\right)^{1/2}$, 
which in turn implies the martingale-coboundary decomposition in $\mathbb L^2$, from which 
the invariance principle could have been already deduced. 

This is why we shall focus on some functionals of a particular linear process: the so-called 
Bernoulli shifts. Let $\left(\varepsilon_k\right)_{k\in \Z}$ by an i.i.d. sequence of random variables 
which take the values $0$ and $1$ with probability $1/2$. 
We define for $x=\sum_{k=1}^{+\infty}2^{-k-1}\varepsilon_{-k}$
\begin{equation}
 T^nx:=\sum_{k=1}^{+\infty}2^{-k-1}\varepsilon_{n-k}
\end{equation}
The map $T$ preserves the Lebesgue measure $\lambda$ on the unit interval endowed with the 
Borel $\sigma$-algebra. Given a function $f\colon [0,1]\to \R$, we would like to 
give some sufficient conditions on the regularity of $f$ which guarantee the invariance principle, 
the law of the iterated logarithm and the $p$-law of large numbers.

\begin{propo}\label{propo:application_WIP_LIL}
 Let $f\colon \Omega\to \R$ be a centered function such that for some $p\in (1,2)$, 
 $t^{p/(p-1)}\lambda\ens{\abs f>t}=0$ and for some $\delta>0$, 
 \begin{equation}
  \int_0^{1}\int_0^1\frac{\abs{f(x)-f(y)}^p}{\abs{x-y}}
  \left(\log \frac 1{\abs{x-y}}\right)^{p-1+\delta}\mathrm dx\mathrm dy<\infty
\mbox{ and }
 \end{equation}
 \begin{equation}
  \int_0^{1}\int_0^1\frac{\abs{\widetilde{f}(x)-
 \widetilde{f}(y)}^{\frac p{p-1}}}{\abs{x-y}}
  \left(\log \frac 1{\abs{x-y}}\right)^{\frac 1{p-1}+\delta}
  \mathrm dx\mathrm dy<\infty,
 \end{equation}
 where $\widetilde{f}(x)=f(x)-f(x/2)/2-f((x+1)/2)/2$.
 
Then $f$ satisfies the invariance principle and the functional law of the iterated logarithm.
\end{propo}

\begin{propo}\label{propo:application_LLN}
 Let $p\in (1,2)$. Let $f\colon \Omega\to \R$ be a centered function such that for some 
 $r\in (p,1)$, $f$ belongs to $\mathbb L^r$ and for some $\delta>0$, 
 \begin{equation}
  \int_0^{1}\int_0^1\frac{\abs{f(x)-f(y)}^q}{\abs{x-y}}
  \left(\log \frac 1{\abs{x-y}}\right)^{q-1+\delta}\mathrm dx\mathrm dy<\infty
\mbox{ and }
 \end{equation}
 \begin{equation}
  \int_0^{1}\int_0^1\frac{\abs{\widetilde{f}(x)-
  \widetilde{f}(y)}^r}{\abs{x-y}}
  \left(\log \frac 1{\abs{x-y}}\right)^{r-1+\delta}
  \mathrm dx\mathrm dy<\infty,
 \end{equation}
 where $q:=\max\ens{1,(p-1)r /(r-1)}$ and $\widetilde{f}(x)=f(x)-f(x/2)/2-f((x+1)/2)/2$.
 Then $f$ satisfies the $p$-strong law of large numbers.
\end{propo}

\section{Proofs} 
 
If $h\colon\Omega\to\R$ is a measurable function, we define 
$M^*(h):=\sup_{N\geqslant 1}N^{-1}\abs{S_N(h)}$.

The following lemma about Birkhoff averages will be used in the proof. 
 
\begin{lem}\label{weak_lp_maximal_function} 
 Let $q>1$ and let $h\colon\Omega\to R $ be a measurable function. 
 \begin{enumerate}[label=(\roman*)]
  \item\label{weak_lq_belong} If $h$ belongs to $\mathbb L^{q,\infty }_0$ 
  then the function $M^*(h)$ belongs to $\mathbb L^{q,\infty }_0 $;
 \item\label{lq_lq_inequality} if $h$ belongs 
 to $\mathbb L^q$, then so does $M^*(h)$.
 \end{enumerate}

\end{lem}

\begin{proof}
 By the maximal ergodic theorem, we have for each positive $t$, 
 \begin{equation}\label{maximal_ergodic_theorem}
  t \cdot\mu\ens{M^*(h) \geqslant t}
  \leqslant \mathbb E\left[\abs h\cdot\mathbf 1
  \ens{M^*(h)  \geqslant t} \right].
 \end{equation}
 \begin{enumerate}[label=(\roman*)]
  \item  The expectation can be bounded by 
 \begin{equation}
  \int_0^{+\infty}\min\ens{\mu(A_t),
  s^{-q}\norm{h\cdot\mathbf 1(A_t)}_{q,\infty}   } \mathrm ds,
 \end{equation}
where 
\begin{equation}
 A_t:=\ens{M^*(h) \geqslant t}. 
\end{equation}
 
Therefore, we infer the bound 
\begin{align}
  \mathbb E\left[\abs h\cdot\mathbf 1(A_t) \right]&\leqslant 
 \mu(A_t)^{1-1/q}\norm{h\mathbf 1(A_t)}_{q,\infty}+ \\
 &+
 \norm{h\cdot\mathbf 1(A_t)}_{q,\infty}^q\int_{\norm{h\cdot\mathbf 1(A_t)}_{q,\infty} 
\mu(A_t)^{-1/q} }^{+\infty}s^{-q}\mathrm ds\nonumber\\
&=\mu(A_t)^{1-1/q}\norm{h\cdot\mathbf 1(A_t)}_{q,\infty}+ \\
&+ \norm{h\cdot\mathbf 1(A_t)}_{q,\infty}^q\frac{(\norm{h\cdot\mathbf 1(A_t)}_{q,\infty} 
\mu(A_t)^{-1/q})^{1-q} }{q-1}\nonumber \\ 
&=\mu(A_t)^{1-1/q}\norm{h\cdot\mathbf 1(A_t)}_{q,\infty}+
\frac{\norm{h\cdot\mathbf 1(A_t)}_{q,\infty}\mu(A_t)^{1-1/q} }{q-1}\\ 
&=\frac q{q-1}\mu(A_t)^{1-1/q}\norm{h\cdot\mathbf 1(A_t)}_{q,\infty}.
\end{align}
 Plugging this into \eqref{maximal_ergodic_theorem}, we obtain 
 \begin{equation}\label{estimate_maximal_function}
  t\cdot\mu\ens{M^*(h)  \geqslant t}^{1/q}
  \leqslant \frac q{q-1}\norm{h\cdot\mathbf 1(A_t)}_{q,\infty},
 \end{equation}
 hence is is enough to prove that 
 \begin{equation}
  \lim_{t\to  +\infty}\sup_{ s\geqslant 0}  s^q\mu\left(\ens{\abs h\geqslant s}\cap 
  A_t\right)=0. 
 \end{equation}
To this aim, fix a positive $\varepsilon$; by assumption, there 
exists a positive real number $s_0$ such that for $s\geqslant s_0$, we have 
$s^q\mu\ens{\abs h\geqslant s}\leqslant \varepsilon$, hence 
\begin{equation}
 \sup_{ s\geqslant 0}  s^q\mu\left(\ens{\abs h\geqslant s}\cap 
  A_t\right)\leqslant \max\ens{\varepsilon,s_0^p\mu(A_t)}, 
\end{equation}
 which is smaller than $\varepsilon$ for $t$ large enough.
 
 \item This follows by multiplying \eqref{maximal_ergodic_theorem} 
 by $t^{q-2}$, integrating over $[0,+\infty)$ with respect to the Lebesgue 
 measure and switching the integrals.
 \end{enumerate}
 
 This concludes the proof of Lemma~\ref{weak_lp_maximal_function}. 
\end{proof}
 
We now give the proofs of the main results, which combine 
Lemma~\ref{weak_lp_maximal_function} with the 
ideas of \cite{MR1893125}.
 
\subsection{Proof of sufficient conditions}\label{sec:proof_sufficient_conditions} 
 
\begin{proof}[Proof of Theorem~\ref{cond_suff_WIP}] 
 In view of Theorem~\ref{equivalence_theorem}, we have to show that
 the sequence $\left(n^{-1/2}\displaystyle
 \max_{1\leqslant j\leqslant n}
 \abs{g\circ T^j}\right)_{n\geqslant 1}$ converges to $0$ in probability.  
 
 Let $\varepsilon$ be a positive fixed number. Let $k,n$ be positive integers 
 such that $k<n$. Denoting 
 $p_n:=\mu\ens{\max_{1\leqslant j\leqslant n}\abs{g\circ T^j}>2\varepsilon n^{1/2}   }$,
 the following estimates take place:
 \begin{align*}
  p_n&\leqslant 
 \mu\ens{\max_{1\leqslant i\leqslant \left[\frac nk\right]+1}
 \max_{ik\leqslant j<(i+1)k}\abs{g\circ T^{ik}}+\abs{g\circ T^j - g\circ T^{ik} } > 
 2\varepsilon n^{1/2}} \\
 &\leqslant \mu\ens{\max_{1\leqslant i\leqslant \left[\frac nk\right]+1}\abs{g\circ T^{ik} }
 >\varepsilon\sqrt n}+ \sum_{i=1}^{\left[\frac nk\right]+1}
\mu\ens{\max_{ik\leqslant j<(i+1)k}\abs{g\circ T^j - g\circ T^{ik} }> 
 \varepsilon n^{1/2}} \\
 &\leqslant \left(\left[\frac nk\right]+1\right)\mu\ens{\abs g>\varepsilon n^{1/2}}+
 \left(\left[\frac nk\right]+1\right)\mu\ens{\max_{0\leqslant j<k}\abs{g\circ T^j-g}>\varepsilon 
 n^{1/2} }\\ 
 &\leqslant \left(\left[\frac nk\right]+1\right)\mu\ens{\abs g>\varepsilon n^{1/2}}+
 \left(\left[\frac nk\right]+1\right)\mu\ens{\max_{0\leqslant j<k}\frac 1j
 \abs{S_j(g-g\circ T)}>\varepsilon 
 \frac{n^{1/2}}k  }.
 \end{align*}
 This yields 
 \begin{equation}
  \label{bound_p_n}
  p_n\leqslant \left(\left[\frac nk\right]+1\right)\mu\ens{\abs g>\varepsilon n^{1/2}}+ 
  \left(\left[\frac nk\right]+1\right)\mu\ens{M^*(g-g\circ T)>\varepsilon \frac{n^{1/2}}k  }.
 \end{equation}
 
 By the assumption on $p$ and $r$, the inequality
\begin{equation} 
\frac{r/2-1}{r-1}= \frac{r-1-r/2}{r-1}=1-\frac{r}{2(r-1)}\geqslant 1-\frac p2  
\end{equation}
takes place, hence we may choose a number $\alpha$ such that 
\begin{equation}\label{cond_alpha}
 1-\frac p2 \leqslant \alpha \leqslant \frac{r/2-1}{r-1}.
\end{equation}
 We now use \eqref{bound_p_n} with $k:=\left[ n^\alpha\right]$. This yields, 
 for some constant $c$ depending only on $p$ and $r$:
\begin{equation}
 p_n\leqslant c\cdot n^{1-\alpha}  \mu\ens{\abs g>\varepsilon n^{1/2}} 
 +c\cdot n^{1-\alpha}\mu\ens{M^*(g-g\circ T)>\varepsilon n^{1/2-\alpha}  },
\end{equation}
 and, by \eqref{cond_alpha}, 
\begin{align}
  p_n&\leqslant cn^{p/2}  \mu\ens{\abs g>\varepsilon n^{1/2}}+ \\ 
  &+cn^{1-\alpha-r(1/2-\alpha)} n^{(1/2-\alpha)r} \mu\ens{M^*(g-g\circ T)
  >\varepsilon n^{1/2-\alpha}  }\nonumber\\
 &=cn^{p/2}  \mu\ens{\abs g>\varepsilon n^{1/2}} +\\
 &+cn^{1-r/2+\alpha(r-1)} n^{(1/2-\alpha)r} 
 \mu\ens{M^*(g-g\circ T)
 >\varepsilon n^{1/2-\alpha}  }\nonumber\\ 
 &\leqslant cn^{p/2}  \mu\ens{\abs g>\varepsilon n^{1/2}}+c n^{(1/2-\alpha)r} 
 \mu\ens{M^*(g-g\circ T)>\varepsilon n^{1/2-\alpha}  }.
\end{align}
 Since $g\in\mathbb L^{p,\infty}_0$ and $g-g\circ T\in\mathbb L^{r,\infty}_0$, 
 we conclude by item \ref{weak_lq_belong} of Lemma~\ref{weak_lp_maximal_function} 
 that the sequence $(p_n)_{n\geqslant 1}$ converges to $0$. 
 
 This concludes the proof of Theorem~\ref{cond_suff_WIP}. 
\end{proof}
  
\begin{proof}[Proof of Corollary~\ref{cor:WIP}]
 Condition \eqref{eq:convergence_E0(Sn)} implies by \cite{MR1198662} that 
 $f$ may be written as $f=m+g-g\circ T$, where $(m\circ T^i)_{i\geqslant 0}$ 
 is a martingale differences sequence with respect to the filtration $(T^{-i}\mathcal M)_{i\geqslant 0}$. 
 Since $f$ is $\mathcal M$-measurable, the function $g$ may be written as $\sum_{i\geqslant 0}
 \mathbb E\left[U^if\mid T\mathcal M\right]$. Thus, by condition \eqref{eq:convergence_E0(Sn)}, 
 we derive that $g\in \mathbb L^{p,\infty}$ and since each term in the series defining 
 $g$ belongs to $\mathbb L^{p,\infty}_0$, we derive that $g$ belongs to $\mathbb L^{p,\infty}_0$. 
 Similarly, by condition \eqref{eq:convergence_(I-U)E0(Sn)}, we infer that 
 $g-g\circ T$ belongs to $\mathbb L^{p/(p-1),\infty}_0$. Since 
 $f\in \mathbb L^{p/(p-1),\infty}_0$, we have $m\in \mathbb L^{p/(p-1),\infty}_0$,
 and accounting the inequality $p/(p-1)>2$, we conclude that $m$ is square integrable. 
 The proof is complete since we showed that $g$ satisfies the conditions of 
 Theorem~\ref{cond_suff_WIP}.
\end{proof}
  
\begin{proof}[Proof of Proposition~\ref{propo:comparison}]
 
 We use the contruction given in \cite{MR2446326}. 
 There exists two independent and $T$-invariant sub-$\sigma$-algebras 
 $\mathcal B$ and $\mathcal C$. We consider a $\mathcal B$-measurable function 
$e_0\colon \Omega\to \ens{-1,1}$ such that $\mu(\ens{e_0=1})=\mu(\ens{e_0=-1})=1/2$, 
and define $e_i:=e_0\circ T^i$, $i\in \Z$, and $\mathcal M:=\mathcal C\vee 
\sigma\ens{e_i,i\leqslant 0}$ (which satisfies $T\mathcal M\subset\mathcal M$). We introduce 
three sequences: $(\theta_k)_{k\geqslant 1}\subset (0,+\infty)$, $(\rho_k)_{k\geqslant 1}
\subset (0,1)$ a decreasing sequence such that $\sum_{k\geqslant 1}\rho_k<1$ and 
an increasing sequence of integers $(N_k)_{k\geqslant 1}$. Once these sequences are 
fixed, we choose a decreasing sequence $(\varepsilon_k)_{k\geqslant 1}\subset (0,1)$ 
such that 
\begin{equation}
 \sum_{k\geqslant 1}\theta_kN_k\varepsilon_k^{1/p}.
\end{equation}
We now consider for each fixed $k\geqslant 1$ a set $A_k\in \mathcal C$ such that 
\begin{enumerate}
 \item the sets $A_k$ are mutually disjoint;
 \item $\left(1-\sum_{i\geqslant 1}\rho_i\right)\frac{\rho_k}2
 \leqslant \mu(A_k)\leqslant \rho_k$ for all $k\geqslant 1$;
 \item for each $k\geqslant 1$ and all 
 $i,j\in \ens{0,\dots,N_k+1}$, $\mu(T^{-i}A_k\Delta T^{-j}A_k)\leqslant \varepsilon_k$.
\end{enumerate}
The existence of such a sequence of sets as well as that of $\mathcal B$ and $\mathcal C$ is explained 
in \cite{MR2446326}.
 
 The function $f$ is defined by 
 \begin{equation}\label{eq:definition_f}
  f=\sum_{k=1}^{+\infty}\theta_ke_{-N_k}\mathbf 1(A_k).
 \end{equation}
Assume that the sequence $(\theta_k)_{k\geqslant 1}$ is increasing and 
$\theta_k\to +\infty$. 
The function $f$ belongs to $\mathbb L^{p/(p-1),\infty}_0$ if 
\begin{equation}
 \lim_{k\to+\infty }\theta_{k+1}^{p/(p-1)}\sum_{i\geqslant k}\rho_i=0.
\end{equation}
Now, we have  
\begin{multline}
 \mathbb E\left[S_n(f)\mid \mathcal M\right]
 =\sum_{k=1}^{+\infty}\theta_k\sum_{i=1}^{\min\ens{n,N_k}}e_{-N_k+i}\mathbf 1(A_k)+\\
 \sum_{k=1}^{+\infty}\theta_k\sum_{i=1}^{\min\ens{n,N_k}}e_{-N_k+i}
 \left(\mathbf 1(T^{-i}A_k\setminus A_k)-\mathbf 1(A_k\setminus T^{-i}A_k)\right),
\end{multline}
and since 
\begin{equation}
 \norm{\sum_{i=1}^{\min\ens{n,N_k}}e_{-N_k+i}
 \left(\mathbf 1(T^{-i}A_k\setminus A_k)-\mathbf 1(A_k\setminus T^{-i}A_k)\right)}_p
 \leqslant N_k\varepsilon_k^{1/p},
\end{equation}
the sequence $\left(\mathbb E\left[S_n(f)\mid \mathcal M\right]\right)_{n\geqslant 1}$ 
converges in $\mathbb L^p$ if 
\begin{equation}\label{eq:condition_counter_example_MCDLp}
 \lim_{m,n\to \infty}
 \mathbb E\abs{\sum_{k=1}^{+\infty}\theta_k\sum_{i=\min\ens{m,N_k}}^{\min\ens{n,N_k}}
 e_{-N_k+i}\mathbf 1(A_k)}^p
 =0.
\end{equation}
Since the family $(A_k)_{k\geqslant 1}$ is disjoint, \eqref{eq:condition_counter_example_MCDLp} 
is equivalent to 
\begin{equation}\label{eq:condition_counter_example_MCDLp_bis}
 \lim_{m,n\to \infty}
 \sum_{k=1}^{+\infty}\theta_k^p\mathbb E\abs{\sum_{i=\min\ens{m,N_k}}^{\min\ens{n,N_k}}
 e_{-N_k+i}}^p\rho_k
 =0,
\end{equation}
which is implied (by Marcinkievicz-Zygmund inequality) by 
\begin{equation}\label{eq:condition_counter_example_MCDLp_finale}
 \sum_{k=1}^{+\infty}\theta_k^pN_k^{p/2}\rho_k.
\end{equation}
Notice also that 
\begin{align*}
 (I-U) \mathbb E\left[S_n(f)\mid \mathcal M\right]&=
 \sum_{k=1}^{+\infty}\theta_k(I-U)
 \sum_{i=1}^{\min \ens{n,N_k}}e_{-N_k+i}\mathbf 1(T^{-i}A_k)\\
 &=\sum_{k=1}^{+\infty}\theta_k\left(e_{-N_k+1}\mathbf 1(T^{-1}A_k)- 
 e_{-N_k+\min \ens{n,N_k}}\mathbf 1(T^{-\min \ens{n,N_k}}A_k)\right),
\end{align*}
hence the sequence $\left((I-U) \mathbb E\left[S_n(f)\mid T\mathcal M\right]\right)_{n\geqslant 1}$
converges in $\mathbb L^{p/(p-1),\infty}$ as long as 
\begin{equation}\label{eq:condition_counter_example_MCDLp'_finale}
 \lim_{k\to +\infty}\theta_{k+1}^{p/(p-1)}\sum_{i\geqslant k}\rho_i=0.
\end{equation}
By Proposition~3 of \cite{MR2446326}, the function $f$ defined by \eqref{eq:definition_f}
satisfies \eqref{eq:Dedecker_and_Rio_condition} if and only if 
$\sum_{k=1}^{+\infty}\theta_k^2\sqrt{N_k}\rho_k<\infty$. 
We thus have to takes sequences $(\theta_k)_{k\geqslant 1}$, $(\rho_k)_{k\geqslant 1}$
and $(N_k)_{k\geqslant 1}$ such that $(\theta_k)_{k\geqslant 1}$ is increasing, 
\eqref{eq:condition_counter_example_MCDLp_finale} and \eqref{eq:condition_counter_example_MCDLp'_finale} 
hold but $\sum_{k=1}^{+\infty}\theta_k^2\sqrt{N_k}\rho_k=+\infty$. Such a selection is possible; for 
example, take 
\begin{equation}
 \theta_k:=\frac{2^{k(p-1)/p}}{(\log k)^{2/p}}, \quad 
 N_k:=\left[\frac{2^{2k(2-p)/p}}{k^{2/p}}\right],\quad  \rho_k:=2^{-k}.
\end{equation}

\end{proof}

\begin{proof}[Proof of Theorem~\ref{cond_suff_LIL}] 
 In view of Theorem~\ref{equivalence_theorem}, we have to prove that  
 the convergence in \eqref{equivalence_LIL} takes place. 
 Let $\alpha\in (0,1)$ be a number which will be specified later. We define 
 \begin{equation}
  m_j:=\sum_{i=1}^{j-1}\left[i^\alpha\right], j\geqslant 1.  
 \end{equation}
 Notice that for some constant $\kappa$ depending only on $\alpha$, 
 we have 
 \begin{equation}\label{bound_m_j} 
  \frac{j^{\alpha +1}}\kappa\leqslant m_j\leqslant \kappa j^{\alpha +1}, j\geqslant 1.
 \end{equation}

 By the Borel-Cantelli lemma, we have to prove the convergence of the series 
 \begin{equation}\label{definition_pj}
   \sum_{j=1}^{+\infty}p_j,\mbox{ with }p_j:=\mu\ens{\max_{0\leqslant i\leqslant 
   \left[j^\alpha\right]}
 \frac 1{\sqrt{\itlog{m_j} } }\abs{g\circ T^{m_j+i} }>\varepsilon  }  
 \end{equation}
  for each positive $\varepsilon$.
 To this aim, we start from the inequalities
\begin{align*}
p_j&\leqslant \mu
 \ens{\frac 1{\sqrt{\itlog{m_j} } }\abs{g\circ T^{m_j} }>\varepsilon/2}+\\
& +\mu\ens{\max_{0\leqslant i\leqslant \left[j^\alpha\right]}
 \frac 1{\sqrt{\itlog{m_j} } }\abs{g\circ T^{m_j+i}-T^{m_j}  }>\varepsilon/2}\\
&= \mu
 \ens{\frac 1{\sqrt{\itlog{m_j} } }\abs{g}>\varepsilon/2}+\\
& +\mu\ens{\max_{0\leqslant i\leqslant \left[j^\alpha\right]}
 \frac 1{\sqrt{\itlog{m_j} } }\abs{g\circ T^i-g}>\varepsilon/2},
\end{align*}
from which we infer 
\begin{multline}\label{estimate_pj_LIL}
p_j \leqslant \mu
 \ens{\frac 1{\sqrt{\itlog{m_j} } }\abs{g}>\varepsilon/2}+\\
 +\mu\ens{\frac 1{\sqrt{\itlog{m_j} } }M^*(g-g\circ T) >
 \frac{\varepsilon}{2\left[j^\alpha\right]} } .
\end{multline}
 \begin{enumerate}[label=(\roman*)]
  \item Assume that $p>r /(r-1)$. Using the definition of $\norm{\cdot}_{p,\infty}$ and 
  inequality \eqref{bound_m_j}, we obtain 
\begin{equation}\label{estimate_pj_LIL2}
  \mu\ens{\frac 1{\sqrt{\itlog{m_j} } }\abs{g}>\varepsilon/2}
 \leqslant c(p,\varepsilon,\alpha)\kappa^{p /2} \norm g_{p,\infty}^p
 \frac 1{j^{(\alpha+1)p/2}  },
 \end{equation}
where $c(p,\varepsilon,\alpha)$ is independent of $j$.
 Using \eqref{bound_m_j} and \eqref{estimate_maximal_function}, we derive 
 \begin{multline}\label{estimate_pj_LIL3}
  \mu\ens{
 \frac 1{\sqrt{\itlog{m_j} } }M^*(g-g\circ T) >
 \frac{\varepsilon}{2\left[j^\alpha\right]} }\leqslant\\ 
 \leqslant c(r,\varepsilon,\alpha) \kappa^{r/2} \norm{g-g\circ T}_{r,\infty}^r
 j^{\alpha r-(\alpha +1)r /2}, 
 \end{multline}
 where $c(r,\varepsilon,\alpha)$ is independent of $j$.
 
 Combining \eqref{estimate_pj_LIL}, \eqref{estimate_pj_LIL2} and 
 \eqref{estimate_pj_LIL3}, we deduce the upper bound 
 \begin{equation}\label{estimate_pj_LIL4} 
  p_j\leqslant c(p,r,\alpha,\varepsilon,g)
  \left(\frac 1{j^{(\alpha+1)p/2}}  + \frac 1{j^{(1-\alpha)r/2}} \right).
  \end{equation}
We have to take $\alpha$ such that 
\begin{equation}\label{choice_of_alpha} 
 (\alpha+1)p/2 >1\mbox{ and } (1-\alpha)r/2>1.
\end{equation}
 This is equivalent to 
 \begin{equation}
  \alpha >2/p-1\mbox{ and }\alpha<1-2/r. 
 \end{equation}
 Since $2/p-1 <1-2/r$, inequalities \eqref{choice_of_alpha} are 
 satisfied and in view of \eqref{estimate_pj_LIL4}, the series 
 defined by \eqref{definition_pj} is convergent for any 
 positive $\varepsilon$. This conclude part 
\ref{strict_inequality} of Theorem~\ref{cond_suff_LIL}.
 
 \item Assume that $p=r /(r-1)$. We pick $\alpha:=2/p-1=1-2/r$. In this case, for some 
 constant $c$ depending only on $p$ and $r$, the inequality 
 \begin{align}
  p_j&\leqslant \mu\ens{\abs{g}> c\varepsilon j^{1/p}} +
  \mu\ens{M^*(g-g\circ T) >
 \varepsilon c j^{-\alpha+(\alpha +1)/2}  }\\ 
 &=\mu\ens{\abs{g}> c\varepsilon j^{1/p}} +
  \mu\ens{M^*(g-g\circ T) >
 \varepsilon c j^{(1-\alpha )/2}  }\\
 &=\mu\ens{\abs{g}> c\varepsilon j^{1/p}} +
  \mu\ens{M^*(g-g\circ T) >\varepsilon c j^{1/r}}
 \end{align}
 takes place.
 By item \ref{lq_lq_inequality} of Lemma~\ref{weak_lp_maximal_function}, 
 we conclude that the series defined by \eqref{definition_pj} is 
 convergent for any positive $\varepsilon$ and 
 this concludes part \ref{equality_case} of 
 Theorem~\ref{cond_suff_LIL}, hence 
 the proof of Theorem~\ref{cond_suff_LIL}.
 \end{enumerate}
\end{proof}

\begin{proof}[Proof of Corollary~\ref{cor:cond_suff_LIL}]
 Like in the proof of Corollary~\ref{cor:WIP}, we derive that 
 $f$ admits a martingale-coboundary decomposition with 
 $g$ in $\mathbb L^p$ and $m, g-g\circ T\in \mathbb L^{p/(p-1)}$. 
 We thus may apply Theorem~\ref{cond_suff_LIL} to conclude that $f$ 
 satisfies the functional law of the iteratd logarithms.
\end{proof}

 \begin{proof}[Proof of Theorem~\ref{cond_suff_LLN}]
  Let us fix a positive $\varepsilon$ and $\alpha\in 
  [1/p,1]$. Let $1\leqslant k<n$ be integers. By similar 
  inequalities which leaded to \eqref{bound_p_n} (we replace the 
  exponent $1/2$ by $\alpha$), we derive 
  \begin{multline}
   \mu\ens{\max_{1\leqslant j\leqslant n}\abs{g-g\circ T^j} >\varepsilon 
   n^\alpha}\leqslant 2\left(\left[\frac nk \right]+1\right)
   \mu \ens{\abs g>\varepsilon n^\alpha /4}  +\\ 
   + 2\left(\left[\frac nk \right]+1\right) 
  \mu\ens{M^*(g-g\circ T) 
  >\frac{\varepsilon n^ \alpha}{ 2k} }. 
  \end{multline}
 Let us choose $k:=\left[n^\beta\right]$, where 
 \begin{equation}\label{choice_of_beta_LLN}
  (p-q)\alpha\leqslant 
 \beta\leqslant  \alpha (r-p) /(r-1)
 \end{equation}
  (the existence of such a $\beta$ 
 is guaranted by the assumptions on $p$, $q$ and $r$). Then 
 it suffices to check that 
 for each positive constant $c$, the series 
 \begin{equation}
 \Sigma_1 :=\sum_{n=1}^{+\infty}n^{p\alpha-1-\beta}\mu\ens{\abs g \geqslant cn^\alpha } 
 \mbox{ and } 
 \end{equation}
 \begin{equation}\label{second_series}
 \Sigma_2 :=\sum_{n=1}^{+\infty}n^{p\alpha-1-\beta}\mu\ens{M^*(g-g\circ T)
 \geqslant cn^{\alpha-\beta}  }
 \end{equation}
  are convergent. The convergence of $\Sigma_1$ is equivalent to the 
  integrability of the function $\abs g^{p-\beta /\alpha}$; this holds 
  since \eqref{choice_of_beta_LLN} implies $q\geqslant p-\beta /\alpha$. 
  
  Note that the second series converges if 
  \begin{equation}
  \mathbb E\left[ \left(M^*(g-g\circ T)\right)^{\frac{p\alpha-\beta }{\alpha-\beta}}
  \right]<+\infty.
  \end{equation}
  Notice that inequality \eqref{choice_of_beta_LLN} implies that 
  $(p\alpha-\beta)/(\alpha-\beta)\leqslant r$, hence we derive
  the convergence of $S_2$
  by item~\ref{lq_lq_inequality} of  
  Lemma~\ref{weak_lp_maximal_function} (with the exponent $
  (p\alpha-\beta )/(\alpha-\beta)>1$ since $p>1$).
   
   This concludes the proof of Theorem~\ref{cond_suff_LLN}.  
 \end{proof}
 
\begin{proof}[Proof of Corollary~\ref{cor:BK}]
 Like in the proof of Corollaries~\ref{cor:WIP} and 
 \ref{cor:cond_suff_LIL}, we derive that $f$ admits a martingale-coboundary decomposition 
 with $g\in \mathbb L^{(p-1)r /(r-1)}$ and $m,g-g\circ T\in \mathbb L^r$. By Theorem~\ref{cond_suff_LLN}, 
 we derive that for each positive $\varepsilon$ and any $\alpha\in[1/p,1]$, \eqref{eq:BK_rewritten} 
 takes place.
\end{proof}

\subsection{Counter-examples}

\begin{proof}[Proof of Theorem~\ref{counter_example}]
We recall the construction given in the proof of Theorem~3 of 
\cite{MR1893125}. We choose a real number $\alpha$ such that 
\begin{equation}\label{choice_of_alpha_counter_example}
 \frac{r-2}{2(r-1)}   <\alpha< 1-\frac p2.
\end{equation}
 This is possible because 
 \begin{multline}
  1-\frac p2 -\frac{r-2}{2(r-1)}=
  \frac 12\left(2-p-\frac{r-1-1}{r-1}  \right) =\\
  =\frac 12\left(1-p+\frac 1{r-1} \right)=\frac 12\left(\frac r{r-1}-p \right)>0. 
 \end{multline}
For each $i\geqslant 1$, we define $n_i:=2^i$ and $k_i:=\left[2^{i\alpha} \right]$. 
By the Rokhlin lemma (see \cite{MR0014222,MR0024503}), one can find a set $A_i\in\mathcal F$ such that 
\begin{equation}\label{disjoint} 
 \mbox{sets }A_i, TA_i,\dots,T^{n_i-1}A_i\mbox{ are pairwise disjoint and }
\end{equation}
 \begin{equation}\label{rokhlin}  
  \mu\left(\bigcup_{j=0}^{n_i-1} T^jA_i  \right)>1/2.
 \end{equation}
In particular, the quantity $\mu(A_i)$ can be bounded as follows:
\begin{equation}\label{bound_mu_Ai}
 \frac 1{2n_i} \leqslant \mu(A_i)\leqslant \frac 1{n_i}.  
\end{equation}

We then define for $i\geqslant 1$, 
\begin{equation}\label{def_gi} 
 g_i:=\frac{\sqrt{ \itlog{n_i}} }{k_i}\left(\sum_{j=1}^{k_i}j\mathbf 1(T^{n_i-j}A_i )
 +\sum_{j=k_i+1}^{2k_i-1}(2k_i-j)\mathbf 1(T^{n_i-j}A_i )\right),   
\end{equation}
 and $g:=\sum_{i=i_0}^{ +\infty}g_i$, where $i_0$ is such that 
 $2k_i<n_i$ for each $i\geqslant i_0$.  By the Borel-Cantelli lemma, 
 since $\mu\ens{g_i\neq 0}\leqslant 2k_i/n_i\leqslant 2^{1-(1-\alpha)i}$, 
 the series which defines $g$ is almost surely convergent.

Since it has been shown in \cite{MR1893125} that the function $f$ 
satisfies neither the invariance principle nor the law of the iterated logarithm, 
it remains to prove that the constructed function $g$ belongs to $\mathbb L^p$ 
and that the coboundary $g-g\circ T$ belongs to $\mathbb L^r$.

By \eqref{disjoint} and \eqref{def_gi}, the equality 
\begin{equation}
 \abs{g_i}^p= \left(\frac{\sqrt{ \itlog{n_i} } }{k_i}\right)^p
 \left(\sum_{j=1}^{k_i}j^p\mathbf 1(T^{n_i-j}A_i )
 +\sum_{j=k_i+1}^{2k_i-1}(2k_i-j)^p\mathbf 1(T^{n_i-j}A_i )\right)
\end{equation}
 takes place, hence integrating and accounting \eqref{bound_mu_Ai}, 
 we derive the estimates
 \begin{align}
 \mathbb E\abs{g_i}^p&\leqslant  \left(\frac{\sqrt{ \itlog{n_i}  }}{k_i}\right)^p
 \left(\sum_{j=1}^{k_i}j^p+\sum_{j=k_i+1}^{2k_i-1}(2k_i-j)^p)\right)\frac 1{n_i}\\
 &\leqslant 2\frac{n_i^{p/2-1}(\log \log n_i)^{p/2}  }{k_i^p}k_i^{p+1}\\ 
 &=2n_i^{p/2-1}(\log \log n_i)^{p/2}k_i, 
\end{align}
hence 
\begin{equation}
 \norm{g_i}_p\leqslant 2^{1/p}n_i^{1/2-1/p}(\log \log n_i)^{1/2}k_i^{1/p}.   
\end{equation}
 By definition of $n_i$ and $k_i$, one can find a constant $c$ depending only 
 on $\alpha$ (hence on $p$ and $r$) such that for $i$ large enough, 
\begin{equation}\label{bound_norm_gi} 
 \norm{g_i}_p\leqslant c\cdot 2^{i(1/2-1/p)}(\log i)^{1/2} \cdot 2^{i\alpha/p}
 =c\cdot (\log i )^{1/2} \cdot 2^{i(\alpha-1+p/2)/p}. 
\end{equation}
By \eqref{choice_of_alpha_counter_example}, the series $\sum_{i=1}^{+\infty} 
(\log i )^{1/2} \cdot 2^{i(\alpha-1+p/2)/p}$ is convergent, and we conclude by 
\eqref{bound_norm_gi} that $g$ belongs to $\mathbb L^p$. 

It is proved in \cite{MR1893125} that by construction, the equality 
\begin{equation}
 \abs{g_i-g_i\circ T}=\frac{\sqrt{\itlog{n_i}} }{k_i}\cdot\mathbf 1
\left(\bigcup_{j=1}^{2k_i}T^{n_i-j}A_i \right)
\end{equation}
holds. By \eqref{disjoint} and \eqref{def_gi}, we have 
\begin{equation}
 \norm{g_i-g_i\circ T}_r  \leqslant \frac{\sqrt{\itlog{n_i}} }{k_i}
 \left(\frac{2k_i}{n_i}\right)^{1/r}, 
\end{equation}
 hence by the definition of $n_i$ and $k_i$, we have for $i$ large enough and 
 a constant $c$ depending only on $\alpha$,
 \begin{equation}
  \norm{g_i-g_i\circ T}_r  \leqslant  c\cdot  2^{i( 1/2-1/r) }2^{i\alpha(1/r-1)}
  (\log i)^{1/2},
 \end{equation}
  from which we infer (by \eqref{choice_of_alpha_counter_example}) 
  the convergence of the series 
  $\sum_{i=1}^{+\infty}\norm{g_i-g_i\circ T}_r$ hence the fact that 
  the function $g-g\circ T$ belongs to $\mathbb L^r$. 
 
 This concludes the proof of Theorem~\ref{counter_example}. 
\end{proof}
 
 \begin{proof}[Proof of Theorem~\ref{counter_example_LLN}]
  The construction is similar to that of the proof of Theorem~\ref{counter_example}. 
  
  For each $i\geqslant 1$, we define $n_i:=2^i$ and $k_i:=\left[2^{i\beta} \right]$, 
  where $\beta$ satisfies 
  \begin{equation}
   \frac{r-p}{p(r-1)}<\beta<1-\frac qp.  
  \end{equation}
 Such a choice is possible since 
 \begin{equation}
  p-q-\frac{r-p}{r-1}>p-(p-1)\frac r{r-1} -\frac{r-p}{r-1} 
  =\frac{p(r-1)-(p-1)r-r+p}{r-1}  =0.
 \end{equation}
 We take a set $A_i\in\mathcal F$ such that \eqref{disjoint} and 
 \eqref{rokhlin} hold. We then define for $i\geqslant 1$, 
\begin{equation}\label{def_gi_LLN} 
 g_i:=\frac{ n_i^{1/p} }{k_i}\left(\sum_{j=1}^{k_i}j\mathbf 1(T^{n_i-j}A_i )
 +\sum_{j=k_i+1}^{2k_i-1}(2k_i-j)\mathbf 1(T^{n_i-j}A_i )\right),   
\end{equation}
 and $g:=\sum_{i=i_0}^{ +\infty}g_i$,  where $i_0$ is such that 
 $2k_i<n_i$ for each $i\geqslant i_0$.
 
 The proof will be complete if we show the following 
 three assertions:
 \begin{enumerate}
  \item the function $g$ belongs to $\mathbb L^q$;
  \item the function $g-g\circ T$ belongs to $\mathbb L^r$;
  \item the sequence $(n^{-1/p} g\circ T^n)_{n\geqslant 1}$ does 
  not converge almost surely to $0$.
 \end{enumerate}
 The first two items follow by completely similar computations as in the 
 proof of Theorem~\ref{counter_example}. 
 To show the last item, we notice that the sequence 
 $(2^{-i/p}\max_{2^i\leqslant l\leqslant 2^{i+1} }
 g\circ T^l)_{i\geqslant 1}$ does not converge to $0$ in 
 probability. To see this, one can note that 
 \begin{multline}\label{eq:proof_CE2_last_equation}
  \mu\ens{2^{-i/p}\max_{2^i\leqslant l\leqslant 2^{i+1} }
 g\circ T^l \geqslant 1}\geqslant \\\mu\ens{2^{-i/p}\max_{2^i\leqslant l\leqslant 2^{i+1} }
 g_i\circ T^l \geqslant 1}\geqslant\mu\left(\bigcup_{j=1}^{n_i-k_i}T^j(A_i)  \right)
 \geqslant \frac 12.
 \end{multline}
 Indeed, since $g_j$ is non-negative, we have $g_i\leqslant g$, which 
 gives the first inequality of \eqref{eq:proof_CE2_last_equation}. 
 For the second, notice that since $n_i=2^i$, 
 \begin{align}
  \ens{2^{-i/p}\max_{0\leqslant l\leqslant 2^{i} }
 g_i\circ T^l \geqslant 1}&=
 \ens{\max_{0\leqslant l\leqslant 2^i}\sum_{j=1}^{k_i}\frac j{k_i}
 \mathbf 1\left(T^{n_i-j}A_i\right)\circ T^l\geqslant 1
 } \\
 &\supset \bigcup_{l=0}^{2^i}\ens{\omega\mid T^l\omega \in T^{n_i-k_i}A_i}
 \supset \bigcup_{j=1}^{n_i-k_i}T^j(A_i).
 \end{align}

 This finishes the proof of Theorem~\ref{counter_example_LLN}.  
 
 \end{proof}
 
\subsection{Applications}

For the context, we refer the reader to Subsection~\ref{sec:applications}.
We define $\mathcal M:=\sigma(\varepsilon_i,i\leqslant 0)$.

The proof of Propositions~\ref{propo:application_WIP_LIL} and 
\ref{propo:application_LLN} will follow from Corollaries~\ref{cor:WIP}, \ref{cor:cond_suff_LIL} 
and \ref{cor:BK} and the following intermediate step. 

\begin{lem}
 Let $q>1$. Then for each centered $f\colon [0,1]\to \R$ and each $n\geqslant 1$, 
 the following inequalities hold:
 \begin{equation}\label{eq:norm_EXn}
  \norm{\mathbb E\left[f\circ T^n\mid \mathcal M\right]}_q^q 
  \leqslant 2^n\int_0^1\int_0^1\mathbf 1\ens{\abs{x-y}\leqslant 2^{-n}}
  \abs{f(x)-f(y)}^p\mathrm dx\mathrm dy
 \end{equation}
\begin{multline}\label{eq:norm_I_U_Xn}
  \norm{(I-U)\mathbb E\left[f\circ T^n\mid \mathcal M\right]}_q^q
  \leqslant\\ 2\int_0^{1}\int_0^1\mathbf 1\ens{\abs{x-y}\leqslant 2^{-n}}
  \abs{\widetilde{f}(x)-\widetilde{f}(y)}^q
  \mathrm dx\mathrm dy<\infty,
 \end{multline}
 where $\widetilde{f}(x)=f(x)-f(x/2)/2-f((x+1)/2)/2$.
\end{lem}
\begin{proof}
Following \cite{MR1782272}, we have for each $x\in [0,1]$, 
\begin{equation}
 \mathbb E\left[f\circ T^n\mid \mathcal M\right](x)
 =2^{-n}\sum_{j=0}^{2^n-1}\int_0^1\left[f\left(\frac{x+j}{2^n}\right)-
 f\left(\frac{y+j}{2^n}\right)\right]\mathrm dy,
\end{equation}
hence by Jensen's inequality, 
\begin{equation}
 \norm{\mathbb E\left[f\circ T^n\mid \mathcal M\right]}_q^q\leqslant 
 2^n\int_0^1\int_0^1\mathbf 1\ens{\abs{x-y}\leqslant 2^{-n}}\abs{f(x)-f(y)}^q
 \mathrm dx\mathrm dy.
\end{equation}

Since 
\begin{equation}
 (I-U)\mathbb E\left[f\circ T^n\mid \mathcal M\right](x)
 =2^{-n}\sum_{j=0}^{2^n-1}\int_0^1\left[\widetilde{f}\left(\frac{x+j}{2^n}\right)-
 \widetilde{f}\left(\frac{y+j}{2^n}\right)\right]\mathrm dy,
\end{equation}
 we prove 
\eqref{eq:norm_I_U_Xn} in a similar way.
\end{proof}

\begin{lem}\label{lem:sufficient_cond_Bernoulli}
 Let $q>1$. If for some positive $\delta$, we have 
 \begin{equation}
  \int_0^{1}\int_0^1\frac{\abs{f(x)-f(y)}^q}{\abs{x-y}}
  \left(\log \frac 1{\abs{x-y}}\right)^{q-1+\delta}\mathrm dx\mathrm dy<\infty,
 \end{equation}
then the series $\sum_{n\geqslant 1}\norm{\mathbb E\left[f\circ T^n\mid \mathcal M\right]}_q$ 
converges.
\end{lem}

A similar sufficient condition can be stated for 
the convergence of the 
series $\sum_{n\geqslant 1}\norm{(I-U)\mathbb E\left[f\circ T^n\mid \mathcal M\right]}_q$.

\begin{proof}[Proof of Lemma~\ref{lem:sufficient_cond_Bernoulli}]
 In view of inequality \eqref{eq:norm_EXn}, we have to prove that 
 \begin{equation}
  \sum_{n=1}^{+\infty}\left(2^n\int_0^1\int_0^1\mathbf 1\ens{\abs{x-y}\leqslant 2^{-n}}\abs{f(x)-f(y)}^q
 \mathrm dx\mathrm dy\right)^{1/q}<+\infty.
 \end{equation}
To this aim, we define $\beta:=(q-1+\delta)/q$ and bound, by Hölder's inequality, 
\begin{equation*}
 \left(\sum_{n=1}^{+\infty}n^{\beta q}
 2^n\int_0^1\int_0^1\mathbf 1\ens{\abs{x-y}\leqslant 2^{-n}}\abs{f(x)-f(y)}^q
 \mathrm dx\mathrm dy\right)\cdot \left(\sum_{n=1}^{+\infty}n^{-\beta q/(q-1)}\right)^{1-1/q}<+\infty
\end{equation*}
and since $\beta q/(q-1)=(q-1+\delta)/(q-1)>1$, it suffices to show that the series 
$\sum_{n=1}^{+\infty}n^{\beta q}
 2^n\int_0^1\int_0^1\mathbf 1\ens{\abs{x-y}\leqslant 2^{-n}}\abs{f(x)-f(y)}^q
 \mathrm dx\mathrm dy$ converges. 
For a fixed $t\in [0,1]$, we have 
\begin{equation}
 \sum_{n=1}^{+\infty }n^{\beta q}
 2^n\mathbf 1\ens{t\leqslant 2^{-n}}=
 \sum_{n=1}^{\log_2(1/t)}n^{\beta q}
 2^n\leqslant \frac 2t\left(\log_2(1/t)\right)^{\beta q},
\end{equation}
from which the convergence of 
$\sum_{n\geqslant 1}\norm{\mathbb E\left[f\circ T^n\mid \mathcal M\right]}_q$ follows.
\end{proof}

\textbf{Acknowledgements.} The author would like to thank anonymous referees 
for their comments, which led to an improvement  of the presentation of the paper.

\def\polhk\#1{\setbox0=\hbox{\#1}{{\o}oalign{\hidewidth
  \lower1.5ex\hbox{`}\hidewidth\crcr\unhbox0}}}\def\rasp{\leavevmode\raise.45ex\hbox{$\rhook$}}
  \def\cftil#1{\ifmmode\setbox7\hbox{$\accent"5E#1$}\else
  \setbox7\hbox{\accent"5E#1}\penalty 10000\relax\fi\raise 1\ht7
  \hbox{\lower1.15ex\hbox to 1\wd7{\hss\accent"7E\hss}}\penalty 10000
  \hskip-1\wd7\penalty 10000\box7} \def\cprime{$'$}
\providecommand{\bysame}{\leavevmode\hbox to3em{\hrulefill}\thinspace}
\providecommand{\MR}{\relax\ifhmode\unskip\space\fi MR }
\providecommand{\MRhref}[2]{%
  \href{http://www.ams.org/mathscinet-getitem?mr=#1}{#2}
}
\providecommand{\href}[2]{#2}


\begin{thebibliography}{BPP15}

\bibitem[BPP15]{1503.05532}
David Barrera, Costel Peligrad, and Magda Peligrad, \emph{On the functional
  {CLT} for stationary {M}arkov {C}hains started at a point}.

\bibitem[BQ15]{benoist_quint}
Y.~Benoist and J-F. Quint, \emph{Central limit theorem for linear groups}, to
  appear in Annals of Probability (2015).

\bibitem[DM07]{MR2743029}
J.~Dedecker and F.~Merlev{\`e}de, \emph{Convergence rates in the law of large
  numbers for {B}anach-valued dependent variables}, Teor. Veroyatn. Primen.
  \textbf{52} (2007), no.~3, 562--587. \MR{2743029 (2011k:60096)}

\bibitem[DR00]{MR1743095}
J{\'e}r{\^o}me Dedecker and Emmanuel Rio, \emph{{On the functional central
  limit theorem for stationary processes}}, Ann. Inst. H. Poincar{\'e} Probab.
  Statist. \textbf{36} (2000), no.~1, 1--34. \MR{1743095 (2001b:60032)}

\bibitem[Dur09]{MR2475604}
Olivier Durieu, \emph{Independence of four projective criteria for the weak
  invariance principle}, ALEA Lat. Am. J. Probab. Math. Stat. \textbf{5}
  (2009), 21--26. \MR{2475604 (2010c:60109)}

\bibitem[DV08]{MR2446326}
Olivier Durieu and Dalibor Voln{\'y}, \emph{{Comparison between criteria
  leading to the weak invariance principle}}, Ann. Inst. Henri Poincar{\'e}
  Probab. Stat. \textbf{44} (2008), no.~2, 324--340. \MR{2446326 (2010a:60075)}

\bibitem[EJ85]{MR780729}
Carl-Gustav Esseen and Svante Janson, \emph{On moment conditions for normed
  sums of independent variables and martingale differences}, Stochastic
  Process. Appl. \textbf{19} (1985), no.~1, 173--182. \MR{780729 (86m:60120)}

\bibitem[HH80]{MR624435}
P.~Hall and C.~C. Heyde, \emph{{Martingale limit theory and its application}},
  Academic Press Inc. [Harcourt Brace Jovanovich Publishers], New York, 1980,
  Probability and Mathematical Statistics. \MR{624435 (83a:60001)}

\bibitem[Kak43]{MR0014222}
Shizuo Kakutani, \emph{Induced measure preserving transformations}, Proc. Imp.
  Acad. Tokyo \textbf{19} (1943), 635--641. \MR{0014222 (7,255f)}

\bibitem[MW00]{MR1782272}
Michael Maxwell and Michael Woodroofe, \emph{Central limit theorems for
  additive functionals of {M}arkov chains}, Ann. Probab. \textbf{28} (2000),
  no.~2, 713--724. \MR{1782272 (2001g:60164)}

\bibitem[Roh48]{MR0024503}
V.~Rohlin, \emph{A ``general'' measure-preserving transformation is not
  mixing}, Doklady Akad. Nauk SSSR (N.S.) \textbf{60} (1948), 349--351.
  \MR{0024503 (9,504d)}

\bibitem[Vol93]{MR1198662}
Dalibor Voln{\'y}, \emph{{Approximating martingales and the central limit
  theorem for strictly stationary processes}}, Stochastic Process. Appl.
  \textbf{44} (1993), no.~1, 41--74. \MR{1198662 (93m:28021)}

\bibitem[Vol06]{MR2239087}
\bysame, \emph{{Martingale approximation of non-stationary stochastic
  processes}}, Stoch. Dyn. \textbf{6} (2006), no.~2, 173--183. \MR{2239087
  (2007g:60047)}

\bibitem[VS00]{MR1893125}
Dalibor Voln{\'y} and Pavel Samek, \emph{{On the invariance principle and the
  law of iterated logarithm for stationary processes}}, {Mathematical physics
  and stochastic analysis ({L}isbon, 1998)}, World Sci. Publ., River Edge, NJ,
  2000, pp.~424--438. \MR{1893125 (2003b:60049)}

\end{thebibliography}
\end{document}